\documentclass[a4paper,12pt]{article}

\usepackage[utf8]{inputenc}
\usepackage[english]{babel}
\usepackage{hyperref}
\usepackage{amsmath,amsthm,enumerate}
\usepackage{enumitem} 
\usepackage{amssymb}
\usepackage{graphicx, color}
\usepackage{epsfig}
\usepackage{tikz}
\usetikzlibrary{calc,decorations.pathmorphing,decorations.text}
\usepackage{subcaption}
\usepackage{refcount}
\usepackage[hmargin=2cm,vmargin=3cm]{geometry}
\usepackage{mathtools, braket}
\usepackage{authblk}
\usepackage{centernot}

\theoremstyle{plain}
\newtheorem{thm}{Thm}[section]
\newtheorem{Question}{Question}

\newtheorem{theorem}[thm]{Theorem}
\newtheorem{lemma}[thm]{Lemma}

\newtheorem{proposition}[thm]{Proposition}

\newtheorem{problem}[thm]{Problem}

\newtheorem{definition}[thm]{Definition}

\newenvironment{proof*}{\noindent\emph{Proof of the claim:}}{\hfill$\Diamond$}

\renewcommand{\pod}[1]{\allowbreak\mathchoice
	{\if@display \mkern 0mu\else \mkern 0mu\fi (#1)}
	{\if@display \mkern 0mu\else \mkern 0mu\fi (#1)}
	{\mkern 1mu(\mathrm{mod}\mkern 4mu #1)}
	{\mkern 0mu(#1)}
}

\usepackage[color=green!30]{todonotes}

\usepackage{caption}
\usetikzlibrary{matrix}
\usetikzlibrary{decorations.markings}
\tikzstyle{vertex}=[circle, draw, fill=black!50,
inner sep=0pt, minimum width=4pt]
\tikzset{->-/.style={decoration={
			markings,
			mark=at position .5 with {\arrow{>}}},postaction={decorate}}}

\tikzstyle{bigblue}=[color=blue, very thick, >=stealth]
\tikzstyle{lightblue}=[color=blue, thin, >=stealth]

\tikzstyle{bigred}=[color=red, very thick, >=stealth]
\tikzstyle{lightred}=[color=red, thin, >=stealth]

\tikzstyle{biggreen}=[color=black!30!green, very thick, >=stealth]
\tikzstyle{lightgreen}=[color=black!30!green,  thin, >=stealth]

\title{Circular $(4-\epsilon)$-coloring of some classes of signed graphs}

\date{\today}

\author[1, 2]{Franti\v{s}ek Kardo\v{s}}
\author[1]{Jonathan Narboni}
\author[3]{\\ Reza Naserasr}
\author[3]{Zhouningxin Wang}

\affil[1]{\small Université de Bordeaux, CNRS,  LaBRI,  F-33400 Talence, France}
\affil[2]{\small Comenius University, Bratislava, Slovakia}
\affil[3]{\small Université de Paris, CNRS, IRIF, F-75006, Paris, France, France\linebreak\linebreak E-mails: \{frantisek.kardos, jonathan.narboni\}@u-bordeaux.fr, \linebreak \{reza, wangzhou4\}@irif.fr}

\date{}
\begin{document}
\baselineskip 0.65cm

\maketitle
	
\abstract{A circular $r$-coloring of a signed graph $(G, \sigma)$ is an assignment $\phi$ of points of a circle $C_r$ of circumference $r$ to the vertices of $(G, \sigma)$ such that for each positive edge $uv$ of $(G, \sigma)$ the distance of $\phi(v)$ and $\phi(v)$ is at least 1 and for each negative edge $uv$ the distance of $\phi(u)$ from the antipodal of $\phi(v)$ is at least 1. The circular chromatic number of $(G, \sigma)$, denoted $\chi_c(G, \sigma)$, is the infimum of $r$ such that $(G, \sigma)$ admits a circular $r$-coloring. 

This notion is recently defined by Naserasr, Wang, and Zhu who, among other results, proved that for any signed $d$-degenerate simple graph $\hat{G}$ we have $\chi_c(\hat{G})\leq 2d$. For $d\geq 3$, examples of signed $d$-degenerate simple graphs of circular chromatic number $2d$ are provided. But for $d=2$ only examples of signed 2-degenerate simple graphs of circular chromatic number close enough to $4$ are given, noting that these examples are also signed bipartite planar graphs.

In this work we first observe the following restatement of the 4-color theorem: If $(G,\sigma)$ is a signed bipartite planar simple graph where vertices of one part are all of degree 2, then $\chi_c(G,\sigma)\leq \frac{16}{5}$. 
Motivated by this observation, we provide an improved upper bound of $ 4-\dfrac{2}{\lfloor \frac{n+1}{2} \rfloor}$ for the circular chromatic number of a signed 2-degenerate simple graph on $n$ vertices and an improved upper bound of $ 4-\dfrac{4}{\lfloor \frac{n+2}{2} \rfloor}$ for the circular chromatic number of a signed bipartite planar simple graph on $n$ vertices. We then show that each of the bounds is tight for any value of $n\geq 4$.}

\section{Introduction}

A \emph{signed graph} $(G, \sigma)$ is a graph $G$ together with a signature $\sigma$ which assigns to each edge of $G$ one of the two signs, either positive or negative. For simplicity we may use $\hat{G}$ to denote a signed graph based on a graph $G$. A key notion in the study of signed graphs is the notion of \emph{switching}, which is to multiply the signs of all the edges of an edge-cut by a $-$. Two signatures on a same graph are said to be \emph{equivalent} if one is a switching of the other. The \emph{sign of a closed walk} $W$ of $(G, \sigma)$ is the product of the signs of all edges of $W$, counting multiplicity, noting that it is invariant under switching. One of the earliest theorems on signed graph is the following. 

\begin{theorem}{\em \cite{Z82a}}\label{Thm:Zaslavski}
Two signatures $\sigma_1$ and $\sigma_2$ on $G$ are equivalent if and only if each cycle $C$ of $G$ has a same sign in $(G, \sigma_1)$ and $(G, \sigma_2)$.
\end{theorem}

The study of coloring and homomorphisms of signed graphs has gained recent attention for various reasons, in particular because it provides a frame for a better connection between minor theory and graph coloring. A \emph{homomorphism} of a signed graph $(G,\sigma)$ to a signed graph $(H,\sigma)$ is a mapping of vertices and edges of $G$ (respectively) to the vertices and edges of $H$ which preserves  adjacencies, incidences, and signs of closed walks. As an application of Theorem~\ref{Thm:Zaslavski}, one can show that this definition is equivalent to finding an equivalent signature $\sigma'$ of $\sigma$ and a mapping of vertices and edges of $G$ to the vertices and edges of $H$ (respectively) which preserves adjacencies, incidences and signs of edges with respect to $\sigma'$ and $\pi$. 

When $G$ and $H$ are simple graphs, then the edge mapping would be implied from the vertex mapping. This would be the case in most of this work. The \emph{core} of a signed graph $(G, \sigma)$ is the smallest subgraph of it to which $(G, \sigma)$ admits a homomorphism. For a proof that core is well defined and unique up to a switch-isomorphism we refer to \cite{NSZ21}.

In the study of homomorphisms of signed graphs, two restrictions stand out: I. restriction to signed graphs $(G, -)$ where all edges are negative, II. restriction to signed bipartite graphs. In the former class of signed graphs, the existence of a homomorphism between two members is solely based on a homomorphism of the graphs they are based on. Thus the homomorphism study in this subclass is the same as the homomorphism study of graphs while a better connection to minor theory can be provided. The restriction on the class of signed bipartite graphs is a focus of this work. It is shown in \cite{NRS15, NSZ21} that the restriction of homomorphism study on this subclass is already as rich as the study of graph homomorphism by the following construction. 

\begin{definition}
Given a graph $G$, the signed (bipartite) graph $S(G)$ is the signed graph obtained from $G$ by replacing each edge $uv$ of $G$ with a negative 4-cycle where $u$ and $v$ are two non-adjacent vertices of it and the other two vertices are new.
\end{definition}

It is then shown in \cite{NRS15} that:

\begin{theorem}
Given graphs $G$ and $H$ there exists a homomorphism of $G$ to $H$ if and only if $S(G)$ maps to $S(H)$.
\end{theorem}

In this work we consider the recent definition of the circular chromatic number of signed graph as defined in \cite{NWZ21}, noting that a similar definition was given earlier in \cite{KS18} but that the two parameters behave differently with respect to fine details.

\begin{definition}
	Given a real number $r$ and signed graph $(G, \sigma)$, a mapping  $\varphi$ of the vertices of $G$ to a circle of circumference $r$ is said to be a \emph{circular $r$-coloring} of $(G, \sigma)$ if for each positive edge $uv$ of $(G, \sigma)$, the distance between $\varphi(u)$ and $\varphi(v)$ is at least  1 and for each negative edge $uv$ of $(G, \sigma)$,  the distance between $\varphi(u)$ and and the antipodal of $\varphi(v)$ is at least 1. The \emph{circular chromatic number} of a signed graph $(G, \sigma)$ is defined as $$\chi_c(G, \sigma) = \inf \{r \ge 1: (G, \sigma) \text{ admits a circular $r$-coloring}\}.$$
\end{definition}


One of the first theorems in the study of circular coloring is the notion of \emph{tight cycle} that is used to prove that the circular chromatic number of a graph is a rational number of the form $\dfrac{p}{q}$ where $p$ is at most the number of vertices. An extension to signed graphs, given in \cite{NWZ21} is as follows.

 \begin{proposition}\label{prop:st}
	Any signed graph $(G, \sigma)$ which is not a forest has a cycle with $s$ positive edges and $t$ negative edges such that $$\chi_c(G, \sigma)= \dfrac{2(s+t)}{2a+t}$$ for some integer $a$. In particular, $\chi_c(G, \sigma) = \dfrac{p}{q}$ for some $p \le 2|V(G)|$.
\end{proposition}

An equivalent definition of circular $r$-coloring of a signed graph is as follows.

\begin{definition}
	A circular $r$-coloring of a signed graph $(G, \sigma)$ is a mapping $f: V(G) \to [0,r)$ such that for each positive edge $uv$, $$1 \le |f(u)-f(v)| \le r-1$$ and for each negative edge $uv$, $$\text{ either } |f(u)-f(v)| \le \frac{r}{2} -1 \text{ or } |f(u)-f(v)| \ge \frac{r}{2} +1.$$
\end{definition}

It follows from this definition that every signed bipartite graph (not necessarily simple) is circular 4-colorable. Simply assign $0$ to vertices of one part and $2$ to the vertices of the other part. However, this should not mislead to underestimating the study of circular chromatic number of signed bipartite graphs. Since $S(G)$ preserves the homomorphism properties of $G$, it is natural to expect that it can be used to determine the circular chromatic number of $G$. This has indeed been proved to be the case in \cite{NWZ21}.

 \begin{proposition}\label{prop:X(G)-->X(S)G))}
Given a simple graph $G$, we have $\chi_c(S(G)) = 4-\dfrac{4}{\chi_c(G)+1}$.
\end{proposition}

Observe that if $G$ is a planar graph then so is $S(G)$. Furthermore, $S(G)$ is a bipartite graph in which vertices of one part are all of degree $2$. Let $\mathcal{SPB}_{2}$ be the class of signed bipartite planar simple graphs for each of which in one part every vertex is of degree at most $2$. It is clear that for each planar graph $G$, $S(G)$ is in $\mathcal{SPB}_{2}$ and that core of each signed bipartite graph in $\mathcal{SPB}_{2}$ is a subgraph of $S(G)$ for some planar graph $G$. 

Combining these observations with Proposition~\ref{prop:X(G)-->X(S)G))}, we have the following reformulation of the Four-Color Theorem.

\begin{theorem}{\em [Four-Color Theorem restated]}
	Every signed graph in $\mathcal{SPB}_{2}$ admits a circular $\dfrac{16}{5}$-coloring.
\end{theorem}

This then naturally leads to two questions, each based on dropping one of the conditions.

\begin{problem}
	What is the best upper bound on the circular chromatic number of signed $2$-degenerate simple graphs?
\end{problem}

\begin{problem}
What is the best upper bound on the circular chromatic number of signed bipartite planar simple graphs?
\end{problem}

In \cite{NWZ21} it is shown that the answer for both questions is $4$. Furthermore, a sequence of signed bipartite $2$-degenerate graphs is built whose circular chromatic number tends to $4$. It is then left as open problem whether one can build an example reaching the exact bound of $4$.

Let $\mathcal{C}_{< 4}$ be the class of signed graphs of circular chromatic number strictly smaller than 4. The questions then are equivalent to ask: 1. Does $\mathcal{C}_{< 4}$ contain the class of all signed 2-degenerate simple graphs? 2. Does $\mathcal{C}_{< 4}$ contain the class of all signed bipartite planar simple graphs?

In this work we answer these questions. In fact, using the number of vertices as a parameter, we provide an improved upper bound for each of the two problems and we show that our bounds are tight. More precisely, we prove the followings.

\begin{theorem}\label{thm:Signed2-degenerate}
	If $(G, \sigma)$ is a signed $2$-degenerate simple graph on $n$ vertices, then $\chi_c(G,\sigma)\leq 4-\dfrac{2}{\lfloor \frac{n+1}{2} \rfloor}.$ Moreover, this upper bound is tight for each value of $n$.
\end{theorem}

\begin{theorem}\label{thm:SignedBipartitePlanar}
	If $(G, \sigma)$ is a signed bipartite planar simple graph on $n$ vertices, then $\chi_c(G,\sigma)\leq 4-\dfrac{4}{\lfloor \frac{n+2}{2} \rfloor}$. Moreover, this upper bound is tight for each value of $n$.
\end{theorem}

The paper is organized as follows. In the next section we
prove Theorem~\ref{thm:Signed2-degenerate}.
In Section~\ref{sec:C<4} we present two graph operations each of which is $\mathcal{C}_{< 4}$-closed. Using this, in Section~\ref{sec:BipartitePlanar}, we prove Theorem~\ref{thm:SignedBipartitePlanar}. Finally in the last section we mention some related problems.

\section{Signed $2$-degenerate simple graphs}\label{sec:2degenerate}

In this section, we first prove the following theorem which, in particular, implies that circular chromatic number of any signed 2-degenerate simple graph is strictly smaller than 4. Then using the notion of tight cycle and Proposition~\ref{prop:st}, we will conclude Theorem~\ref{thm:Signed2-degenerate}.

\begin{theorem}\label{thm:2-vertex}
	Let $\hat{G}$ be a signed simple graph with a vertex $w$ of degree $2$. If the signed graph $\hat{G}-w$ has circular chromatic number strictly less than $4$, then $\hat{G}$ also has circular chromatic number strictly less than $4$.
\end{theorem}

\begin{proof}
	Let $\hat{G}$ be a minimum counterexample to the theorem. Then it follows immediately that $G$ is connected and has no vertex of degree 1. Let $u$ and $v$ be the two neighbors of $w$. Since circular chromatic number is invariant under switching, and without loss of generality, we may assume both $uw$ and $vw$ are positive edges in $\hat{G}$. 
	
	Let $\hat{G}'=\hat{G}-w$ and let $\epsilon$ be a positive real number smaller than $2$, such that $\hat{G}'$ admits a circular $(4-\epsilon)$-coloring. Let $C$ be the circle of circumference $4-\epsilon$. 

    By rotational symmetries of the circle we can assume that $\varphi(u)=0$. Then considering symmetries along the diameters of the circle, in particular the one that contains 0, we may assume $ \varphi(v) \geq 2-\frac{\epsilon}{2} $. Furthermore, we may assume $\varphi(v) < 2$ as otherwise we can complete $\varphi$ to a coloring of $\hat{G}$ simply by setting $\varphi(w)=1$.

    Our aim is to present a circular $(4-\frac{\epsilon}{4})$-coloring $\psi$ of $\hat{G}$. To this end, first we do a uniform scaling of the circle $C$ to a circle $C'$ to get a circular $(4-\frac{\epsilon}{2})$-coloring $\varphi'$ of $\hat{G}'$.
	More precisely $\varphi': V(\hat{G}')\to [0, 4-\frac{\epsilon}{2})$ is defined as follows. $$\varphi'(x)=\dfrac{4-\frac{\epsilon}{2}}{4-\epsilon}\varphi(x).$$
	
	The mapping $\varphi'$ has the property that for a positive edge $xy$ the points $\varphi'(x)$ and $\varphi'(y)$ are at distance (on $C'$) at least $1+\frac{\epsilon}{8-2\epsilon}$ and that the same holds for the distance between $\varphi'(x)$ and the antipodal of $\varphi'(y)$ whenever $xy$ is a negative edge. Observe that $\varphi'(u)=0$ and $\varphi'(v)\ge 2-\frac{\epsilon}4$.
	
	Next we introduce a circular $(4-\frac{\epsilon}{4})$-coloring of $\hat{G}'$ by inserting an interval of length $\frac{\epsilon}{4}$ inside $C'$ to obtain a circle $C''$ of circumference $4-\frac{\epsilon}{4}$. Assuming this interval is inserted at point $1-\frac{\epsilon}{8}$ of $C'$, the new coloring $\psi$ of $\hat{G}'$ is defined as follows.
	
	$$\psi(x) = \begin{cases} 
		\varphi'(x), &\text{ if $\varphi(x)< 1-\dfrac{\epsilon}{8}$}, \\ \cr
		\varphi'(x)+\dfrac{\epsilon}{4}, &\text{ if $\varphi(x)\geq 1-\dfrac{\epsilon}{8}$}. 
	\end{cases}
	$$
		
	We need to verify that $\psi$ is a circular coloring of $\hat{G}'$. For a positive edge $xy$, it's immediate to see that the distance of $\psi(x)$ and $\psi(y)$ is at least $1$, because in changing $C'$ to $C''$ the distance between two points does not decrease. For a negative edge $xy$, we note that since the diameter of the circle is changed, the antipodal of each point is shifted by $\frac{\epsilon}{8}$. To be more precise, if $a$ is a point of circle $C'$ with $a_1$ as its antipodal, and $a'$ and $a'_1$ are the images of these points at $C''$ after inserting an interval of length $\frac{\epsilon}{4}$, the antipodal of $a'$ on $C''$ is at distance $\frac{\epsilon}{8}$ from $a'_1$ (see Figure~\ref{fig:C'} and \ref{fig:C''}). Since in $C'$ the distance between $\varphi'(x)$ and the antipodal of $\varphi'(y)$ is at least $1+\frac{\epsilon}{8-2\epsilon}$, even after this shift of $\frac{\epsilon}{8}$ the distance between $\psi(x)$ and the antipodal of $\psi(y)$ is at least 1 and, therefore, $\psi$ is a circular $(4-\frac{\epsilon}{4})$-coloring. 

	\begin{figure}[ht]
		\centering 
		\begin{minipage}[t]{.48\textwidth}
			\centering 
			\begin{tikzpicture}
				[scale=.4]
				\draw [line width=.05mm, black] (0,0) circle (4cm);
				\draw[rotate=20] (0,4)  node[line width=0mm, circle, fill, inner sep=1.2pt, label=above: \scriptsize{$a$} ](a) {};
				\draw[rotate=0] (0,4)  node[line width=0mm, circle, fill, inner sep=1.2pt, label=below: \scriptsize{$0$} ] {};
				\draw[rotate=20] (0,-4)  node[line width=0mm, circle, fill, inner sep=1.2pt, label=below: \scriptsize{$a_1$} ](a1) {};
				\draw [dashed, line width=0.2mm, black] (a) -- (a1);
			\end{tikzpicture}
			\caption{Circle $C'$ with $r'=4-\dfrac{\epsilon}{2}$}
			\label{fig:C'}
		\end{minipage}
		\begin{minipage}[t]{.48\textwidth}
			\centering 
			\begin{tikzpicture}
				[scale=.4]
				\draw [line width=.06mm, black] (0,0) circle (4cm);
				\draw[rotate=0] (0,4)  node[line width=0mm, circle, fill, inner sep=1.2pt, label=below: \scriptsize{$0$} ] {};
				\draw[rotate=20] (0,4)  node[line width=0mm, circle, fill, inner sep=1.2pt, label=above: \scriptsize{$a'$}] (a'){};
				\draw[rotate=20] (0,-4)  node[line width=0mm, circle, fill, inner sep=1.2pt, label={below right}: \scriptsize{$\bar{a}'$}] (a1){};
				\draw[rotate=10] (0,-4)  node[line width=0mm, circle, fill, inner sep=1.2pt, label=below: \scriptsize{$a'_1$}] (a'1){};
				\draw [dashed, line width=0.2mm, black] (a') -- (a1);
				\draw [black, line width=.6mm, domain=6:-6] plot ({4*cos(\x)}, {4*sin(\x)}) node [left] {$\frac{\epsilon}{4}$};
					
			\end{tikzpicture}
			\caption{Circle $C''$ with $r''=4-\dfrac{\epsilon}{4}$}
			\label{fig:C''}
		\end{minipage}
	\end{figure}

	Finally, as $\psi(u)=0$ and $\psi(v)\ge 2$, we may complete the circular $(4-\frac{\epsilon}{4})$-coloring $\psi$ of $\hat{G}'$ to $\hat{G}$  simply by setting $\psi(w)=1$.
\end{proof}

We observe that in this proof for two vertices $x$ and $y$ of $\hat{G}-w$ if we have $\varphi(x)=\varphi(y)$, then we have $\psi(x)=\psi(y)$.

From the statement of this theorem, it follows immediately that every signed $2$-degenerate simple graph admits a $(4-\epsilon)$-coloring for some positive real number $\epsilon$. Next we use the notion of tight cycle to give a precise upper bound in terms of the number of vertices.

\begin{theorem}\label{thm:2-degenerate}
	For any signed $2$-degenerate simple graph $(G, \sigma)$ on $n$ vertices, we have:
	\begin{itemize}
		\item For each odd value of $n$, $\chi_c(G, \sigma)\leq  4-\dfrac{4}{n+1}$,
		\item For each even value of $n$, $\chi_c(G, \sigma)\leq  4-\dfrac{4}{n}$.
	\end{itemize}
\end{theorem}

\begin{proof}
As stated in Proposition~\ref{prop:st}, we know that $\chi_c(G, \sigma)=\frac{p}{q}$ where $p$ is twice the length of a cycle in $G$. Thus $p$ is an even integer satisfying $p\leq 2n$. Since $\chi_c(G, \sigma)<4$ we have $\frac{p}{q}<4$, in other words, $p<4q$. As $p$ and $q$ are integers, and moreover $p$ is an even integer, we have $p\leq 4q-2$. Therefore, $\chi_c(G, \sigma)\leq \frac{4q-2}{q}=4-\frac{2}{q}$. On the other hand $ \chi_c(G, \sigma)\leq \frac{2n}{q}$. 

For a fixed $n$, the sequence $(\frac{2n}{q})_{q\in \mathbb{N}}$ is decreasing, whereas the sequence $(4-\frac{2}{q})_{q\in \mathbb{N}}$ is increasing.
It is easy to check that
\[
\max_{q\in\mathbb{N}} \min\left\{ \frac{2n}{q}, 4 - \frac{2}{q} \right\} = \begin{cases}
4-\frac{4}{n+1} & \textrm{for $q=\frac{n+1}{2}$ if $n$ is odd,} \\
4-\frac{4}{n} & \textrm{ for $q=\frac{n}{2}$ if $n$ is even.} \hfil \qedhere \\
\end{cases} 
\]
\end{proof}

Next we show that the bound in Theorem~\ref{thm:2-degenerate} is tight for each value of $n\geq 1$. We construct a sequence of signed $2$-degenerate simple graphs $\Omega_i$ reaching the bound for $n=2i+1$. For even values of $n$, it would be enough to add an isolated vertex to $\Omega_i$.

Let $\Omega_1=(K_3, +)$, that is the complete graph on three vertices with all edges being positive. Let $v_1, v_2, v_3$ be its vertices. Starting with $\Omega_1$, we define the sequence $\Omega_i$ of signed graphs as follows. Given $\Omega_i$ on vertices $v_1, v_2, \ldots, v_{2i+1}$, we first add a vertex $v_{2i+2}$ which a copy of $v_{2i+1}$, i.e., it sees each of the two neighbors of $v_{2i+1}$ with edges of the same sign. Then we add a new vertex $v_{2i+3}$ which is joined to $v_{2i+1}$ and $v_{2i+2}$, to one with a negative edge and to the other with a positive edge. Observe that $\Omega_i$ has $2i+1$ vertices and is $2$-degenerate. The elements $\Omega_{2}, \Omega_{3}$ and $\Omega_{4}$ of the sequence are illustrated in Figure~\ref{fig:Omega2}, \ref{fig:Omega3}, \ref{fig:Omega4} respectively.

\begin{figure}[htbp]
	\centering
	\begin{minipage}[t]{.25\textwidth}
		\centering
		\begin{tikzpicture}
			[scale=.45]
		
\draw[rotate=0] (2,0)  node[circle, draw=black!80, inner sep=0mm, minimum size=2mm] (v1){\scriptsize ${v_{_{1}}}$};

\draw[rotate=0] (-2,0)  node[circle, draw=black!80, inner sep=0mm, minimum size=2mm] (v2){\scriptsize ${v_{_{2}}}$};

\draw[rotate=0] (0,2)  node[circle, draw=black!80, inner sep=0mm, minimum size=2mm] (v3){\scriptsize ${v_{_{3}}}$};

\draw[rotate=0] (0,-2)  node[circle, draw=black!80, inner sep=0mm, minimum size=2mm] (v4){\scriptsize ${v_{_{4}}}$};

\draw[rotate=0] (4,0)  node[circle, draw=black!80, inner sep=0mm, minimum size=2mm] (v5){\scriptsize ${v_{_{5}}}$};

\draw[rotate=0] (0,-4)  node[circle, draw=white, inner sep=0mm, minimum size=2mm] (v8){};

\draw  [dashed, line width=0.5mm, red] (v3) -- (v5);
\draw  [line width=0.5mm, blue] (v1) -- (v2);
\draw  [line width=0.5mm, blue] (v4)--(v1) -- (v3) -- (v2) -- (v4) -- (v5);
			
		\end{tikzpicture}
		\caption{$\Omega_2$}
		\label{fig:Omega2}     
	\end{minipage}
	\begin{minipage}[t]{.28\textwidth}
		\centering
  \begin{tikzpicture}
	[scale=.45]
\draw[rotate=0] (2,0)  node[circle, draw=black!80, inner sep=0mm, minimum size=2mm] (v1){\scriptsize ${v_{_{1}}}$};

\draw[rotate=0] (-2,0)  node[circle, draw=black!80, inner sep=0mm, minimum size=2mm] (v2){\scriptsize ${v_{_{2}}}$};

\draw[rotate=0] (0,2)  node[circle, draw=black!80, inner sep=0mm, minimum size=2mm] (v3){\scriptsize ${v_{_{3}}}$};

\draw[rotate=0] (0,-2)  node[circle, draw=black!80, inner sep=0mm, minimum size=2mm] (v4){\scriptsize ${v_{_{4}}}$};

\draw[rotate=0] (4,0)  node[circle, draw=black!80, inner sep=0mm, minimum size=2mm] (v5){\scriptsize ${v_{_{5}}}$};

\draw[rotate=0] (-4,0)  node[circle, draw=black!80, inner sep=0mm, minimum size=2mm] (v6){\scriptsize ${v_{_{6}}}$};

\draw[rotate=0] (0,4)  node[circle, draw=black!80, inner sep=0mm, minimum size=2mm] (v7){\scriptsize ${v_{_{7}}}$};

\draw[rotate=0] (0,-4)  node[circle, draw=white, inner sep=0mm, minimum size=2mm] (v8){};

\draw  [dashed, line width=0.5mm, red] (v6) -- (v3) -- (v5) -- (v7);
\draw  [line width=0.5mm, blue] (v1) -- (v2);
\draw  [line width=0.5mm, blue] (v7) -- (v6) -- (v4) -- (v1) -- (v3) -- (v2) -- (v4) -- (v5);

		\end{tikzpicture}
		\caption{$\Omega_3$}
		\label{fig:Omega3}
	\end{minipage}
	\begin{minipage}[t]{.32\textwidth}
		\centering
		\begin{tikzpicture}
			[scale=.45]
\draw[rotate=0] (2,0)  node[circle, draw=black!80, inner sep=0mm, minimum size=2mm] (v1){\scriptsize ${v_{_{1}}}$};

\draw[rotate=0] (-2,0)  node[circle, draw=black!80, inner sep=0mm, minimum size=2mm] (v2){\scriptsize ${v_{_{2}}}$};

\draw[rotate=0] (0,2)  node[circle, draw=black!80, inner sep=0mm, minimum size=2mm] (v3){\scriptsize ${v_{_{3}}}$};

\draw[rotate=0] (0,-2)  node[circle, draw=black!80, inner sep=0mm, minimum size=2mm] (v4){\scriptsize ${v_{_{4}}}$};

\draw[rotate=0] (4,0)  node[circle, draw=black!80, inner sep=0mm, minimum size=2mm] (v5){\scriptsize ${v_{_{5}}}$};

\draw[rotate=0] (-4,0)  node[circle, draw=black!80, inner sep=0mm, minimum size=2mm] (v6){\scriptsize ${v_{_{6}}}$};

\draw[rotate=0] (0,4)  node[circle, draw=black!80, inner sep=0mm, minimum size=2mm] (v7){\scriptsize ${v_{_{7}}}$};

\draw[rotate=0] (0,-4)  node[circle, draw=black!80, inner sep=0mm, minimum size=2mm] (v8){\scriptsize ${v_{_{8}}}$};

\draw[rotate=0] (6,0)  node[circle, draw=black!80, inner sep=0mm, minimum size=2mm] (v9){\scriptsize ${v_{_{9}}}$};
		
\draw  [dashed, line width=0.5mm, red] (v6) -- (v3) -- (v5) -- (v7) -- (v9);
\draw  [dashed, line width=0.5mm, red] (v8) -- (v5);
\draw  [line width=0.5mm, blue] (v1) -- (v2);
\draw  [line width=0.5mm, blue] (v7) -- (v6) -- (v4) -- (v1) -- (v3) -- (v2) -- (v4) -- (v5);
\draw  [line width=0.5mm, blue] (v6) -- (v8) -- (v9);
		\end{tikzpicture}
		\caption{$\Omega_4$}
		\label{fig:Omega4}
	\end{minipage}
\end{figure}

\begin{proposition}
Given a signed graph $\Omega_{i}$ as defined above, we have $$\chi_c(\Omega_i)=4-\dfrac{4}{|V(\Omega_i)|+1}.$$
\end{proposition}

\begin{proof}
We prove by induction a slightly stronger claim. Let $r_{_{i}}=4-\frac{2}{i+1}$. We claim that $\chi_c(\Omega_i)=r_{_{i}}$ and, moreover, in any circular $r_{_{i}}$-coloring of $\Omega_{i}$ the tight cycle is a Hamiltonian cycle.

The case $i=1$ of this claim is immediate. That $\chi_c(\Omega_i)\leq r_{_{i}}$ follows from Theorem~\ref{thm:2-degenerate}. To show that $\chi_c(\Omega_i)\geq r_{_{i}}$, it is enough to show that $\Omega_i$ is not $r_{_{i-1}}$-colorable, because there are no rational numbers between $r_{_{i-1}}$ and $r_{i}$ with a numerator at most $2(2i+1)$. To this end, and toward a contradiction, assume $\psi$ is a circular $r_{_{i-1}}$-coloring of $\Omega_i$. We claim that $\psi(v_{2i-1})=\psi(v_{2i})$. That is because $\psi$ is also a circular $r_{_{i-1}}$-coloring of $\Omega_{i-1}$, and in any such a coloring the tight cycle (of $\Omega_{i-1}$) is a Hamilton cycle. As $v_{2i-1}$ is of degree 2 in $\Omega_{i-1}$ and $v_{2i}$ is a copy of $v_{2i-1}$, we must have $\psi(v_{2i-1})=\psi(v_{2i})$. But then to complete the circular $r_{_{i-1}}$-coloring to $v_{2i+1}$ we must have a point on the circle which is at distance at least 1 from both $\psi(v_{2i-1})$ and its antipodal. But that is only possible if the circumference of circle used for coloring is at least 4. Thus $\chi_c(\Omega_i)=\frac{4i+2}{i+1}$. We then observe that $gcd(4i+2, i+1)=1$ when $i$ is even and $gcd(4i+2, i+1)=2$ when $i$ is odd. Hence any tight cycle of $\Omega_{i}$ in a circular $\frac{4i+2}{i+1}$-coloring is a Hamilton cycle, completing the proof.
\end{proof}

\section{$\mathcal{C}_{<4}$-closed operations}\label{sec:C<4}

Recall that $\mathcal{C}_{< 4}$ is the class of signed graphs of circular chromatic number strictly smaller than 4. In this section, we present two graph operations that preserve membership in this class.

We first observe that Theorem~\ref{thm:2-vertex} could also be viewed as an operation that preserves membership in this class: For each $(G, \sigma)\in \mathcal{C}_{<4}$ and any pair of distinct vertices $x$ and $y$ of $(G, \sigma)$, if we add a vertex $u$ and join it to $x$ and $y$ with edges of arbitrary signs, then the resulting signed graph is also in $\mathcal{C}_{<4}$.

A slight modification and generalization of this one is based on the following notation. Let $(G, \sigma)$ be a signed graph and let $u$ be a vertex of $(G, \sigma)$. We define $F_u(G, \sigma)$ to be the signed graph obtained from $(G, \sigma)$ by contracting all the edges incident to $u$ and keeping signs of all other edges as it is. One could easily observe that for (switching) equivalent signatures $\sigma$ and $\sigma'$, the signed graphs $F_u(G, \sigma)$ and $F_u(G, \sigma')$ might not be switching equivalent.

\begin{theorem}\label{thm:F_u}
Given a signed graph $(G, \sigma)$ and a vertex $u$ of $(G, \sigma)$, if $\chi_c(F_u(G, \sigma))< 4$, then $\chi_c(G, \sigma)< 4$.
\end{theorem}

As $F_u(G, \sigma)$ and $F_u(G, \sigma')$ might not be switching equivalent even if $\sigma$ and $\sigma'$ are switching equivalent, in applying this theorem it is important to choose a suitable signature (switching equivalent to $\sigma$). In particular, if two neighbors of $u$, say $x$ and $y$, are adjacent with a positive edge, then $F_u(G, \sigma)$ will have a positive loop and
so
its circular chromatic number is $\infty$. Similarly, if two neighbors of $u$ have another common neighbor $v$ which sees one with a positive edge and the other with a negative edge, then $F_u(G, \sigma)$ has a digon and does not belong to $\mathcal{C}_{<4}$.

The proof of this theorem is quite similar to the proof of Theorem~\ref{thm:2-vertex}. We consider a circular $(4-\epsilon)$-coloring of $F_u(G, \sigma)$. Then we consider a corresponding coloring on $(G-u, \sigma)$ noting that all neighbors of $u$ are colored with a same color. We then modify the coloring as in the proof of Theorem~\ref{thm:2-vertex} to find a color for $u$. We leave the details to the reader.

\medskip
Next we define an edge-operation which also preserves membership in $\mathcal{C}_{<4}$.
Let $\hat{G}$ be a signed graph with a positive edge $uv$. We define $F_{uv}(\hat{G})$ to be the signed graph obtained from $\hat{G}$ as follows. First we add a copy $u'$ of $u$, that is to say for every neighbor $w$ of $u$ we join $u'$ to $w$ with an edge which is of the same sign as $uw$. Similarly, we add a copy $v'$ of $v$. Then we add two more vertices $x$ and $y$ with the following connections: $xu$, $yv$ as positive edges and $xu'$, $yv'$, and $xy$ as negative edges. See Figure~\ref{Fig:Operation}.

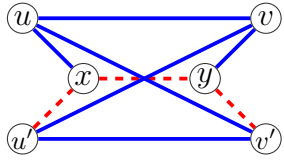
\begin{figure}[htbp]
	\centering
	\begin{tikzpicture}
		[scale=.4]

\draw[rotate=0] (-2,0)  node[circle, draw=black!80, inner sep=0mm, minimum size=4mm] (x){$x$};

\draw[rotate=0] (2,0)  node[circle, draw=black!80, inner sep=0mm, minimum size=4mm] (y){$y$};

\draw[rotate=0] (-4,2)  node[circle, draw=black!80, inner sep=0mm, minimum size=4mm] (u){$u$};

\draw[rotate=0] (-4,-2)  node[circle, draw=black!80, inner sep=0mm, minimum size=4mm] (u'){\footnotesize $u'$};

\draw[rotate=0] (4,2) node[circle, draw=black!80, inner sep=0mm, minimum size=4mm] (v){$v$};

\draw[rotate=0] (4,-2)  node[circle, draw=black!80, inner sep=0mm, minimum size=4mm] (v'){\footnotesize $v'$};

	   \draw  [dashed, line width=0.5mm, red] (u') -- (x) -- (y) -- (v');
	   \draw[line width=0.5mm, blue] (v)  -- (u);
	   \draw[line width=0.5mm, blue] (v')  -- (u');
	   \draw[line width=0.5mm, blue] (v')  -- (u) -- (x);
	   \draw[line width=0.5mm, blue] (u') -- (v) -- (y);
	\end{tikzpicture}
	\caption{The operation $F_{uv}$}
	\label{Fig:Operation}
\end{figure}

We prove that $\mathcal{C}_{<4}$ is closed under the operation $F_{uv}$. 

\begin{theorem}\label{thm:F_uv}
Given a signed graph $\hat{G}$ and a positive edge $uv$ of $\hat{G}$, if $\chi_{c}(\hat{G})< 4$, then  $\chi_{c}(F_{uv}(\hat{G}))< 4$.
\end{theorem}

\begin{proof}
Let $\varphi$ be a circular $(4-\epsilon)$-coloring of $\hat{G}$ for a positive real number $\epsilon$. We assume, without loss of generality, that $\varphi(u)=0$ and $ 1\leq \varphi(v) \leq 2-\frac{\epsilon}{2}$. 
We view $\varphi$ as a partial coloring of $F_{uv}(\hat{G})$. 
Our goal would be to modify $\varphi$ to a circular $(4-\frac{\epsilon}{4})$-coloring of $\hat{G}$ in such a way that we can extend it on the four new vertices $u'$, $v'$, $x$ and $y$. This would be done in two steps. We first scale the circle to increase its circumference by $\frac{\epsilon}{2}$ and then insert an interval of length $\frac{\epsilon}{4}$ into the circle after which we must modify images of some vertices.  The details are as follows.

First we define the $(4-\frac{\epsilon}{2})$-coloring $\varphi'$ as follows:
	$$\varphi'(w) = \frac{4-\frac{\epsilon}{2}}{4 - \epsilon}\varphi(w).$$

This was the scaling part. Let $\gamma=\frac{4-\frac{\epsilon}{2}}{4-\epsilon}$ and note that $$\gamma=1+\frac{\epsilon}{8-2\epsilon}>1+\frac{\epsilon}{8}>1.$$ Next we define $\psi$ and show that it is a circular $(4-\frac{\epsilon}{4})$-coloring of $F_{uv}(\hat{G})$. On the vertices of $\hat{G}$, the mapping $\psi$ is defined as follows.
$$\psi(w) = \begin{cases} 
		\varphi'(w), &\text{ if $\varphi'(w)< 1-\dfrac{\epsilon}{8}$}, \cr
		\varphi'(w) + \dfrac{\epsilon}{4}, & \text{ otherwise.} \cr
\end{cases}$$

It is then extended to the remaining four vertices by: $$\psi(u')=\frac{\epsilon}{8}, \, \psi(v')= \varphi'(v) +\frac{\epsilon}{8}, \, \psi(x)=1, \text{ and } \psi(y)=\varphi'(v)+\frac{\epsilon}{4}-1.$$

We need to show that $\psi$ is a circular $(4-\frac{\epsilon}{4})$-coloring of $F_{uv}(\hat{G})$.
Restriction of $\psi$ on $\hat{G}$ is a circular coloring because if $w_1$ and $w_2$ are two adjacent vertices of $\hat{G}$, then in coloring $\varphi$ either (1) $\varphi(w_1)$ and $\varphi(w_2)$ are at distance at least 1 or (2) $\varphi(w_1)$ and $\overline{\varphi(w_2)}$ are at distance at least 1. This distance then is increased to at least $\gamma$ in $\varphi'$. Then in defining $\psi$ based on $\varphi'$ either the distance remains the same, or it increases by $\frac{\epsilon}{4}$, or one end moves by a value of $\frac{\epsilon}{8}$. Thus, in all the cases the resulting distance is still larger than 1. 

It remains to consider the connection to and among new vertices, $u'$, $v'$, $x$, and $y$. By the definition of $\psi$, the five edges incident to $x$ or $y$ are satisfying the conditions of circular coloring. It remains to verify the condition for edges incident with $u'$ and $v'$ but not incident with $x$ or $y$. 

We first consider the edges incident with $u'$. Recall that $u'$ is a copy of $u$. Let $w$ be a neighbor $u$ in $\hat{G}$. Based on the sign of $wu'$ we consider two cases. 

\begin{itemize}
\item $wu'$ is a positive edge. 

We need to show that the distance between $\psi(w)$ and $\psi(u')$ is at least 1. Using the definition of circular coloring based on the circle, we consider both clockwise and anticlockwise distances on the circle. The anticlockwise path of the circle from $u'$ to $w$ passes through $u$ and since $u$ and $w$ are already proved to be at distance at least 1, the anticlockwise distance from $\psi(u')$ to $\psi(w)$ is larger than 1. For the clockwise direction, since $uw$ is a positive edge we have $\varphi (w)\geq 1$. Thus, by the definition of $\psi$, we $\psi(w)=\varphi'(w)+\frac{\epsilon}{4}$ whereas $\psi(u')=\frac{\epsilon}{8}$. Therefore, the clockwise distance of $\psi(u')$ and $\psi(w)$ is larger than the clockwise distance of $\varphi(u)=0$ and $\varphi'(w)$ which is at least 1.

\item $wu'$ is a negative edge. 

In circular $(4-\epsilon)$-coloring $\varphi$ the distance of $\varphi(u)$ and $\varphi(w)$ is at most $1-\frac{\epsilon}{2}$. Again we consider two possibilities depending on if the distance is obtained in clockwise direction starting from $0$ or anticlockwise. For clockwise direction, we observe that $\varphi'(w)< 1-\frac{\epsilon}{8}$. Thus in defining $\psi$ the distance of $\psi(u)$ and $\psi(w)$ remains the same as the distance of $\varphi'(u)$ and $\varphi'(w)$, and the distance of $\psi(u')$ to $\psi(w)$ is actually shorter. If the distance of $\varphi(w)$ and $\varphi(u)$ is obtained on anticlockwise direction starting from $0$, then this distance is at most $1- \frac{\epsilon}{2}$. Therefore, the distance of $\varphi'(w)$ and $\varphi(u)$ is at most $(1- \frac{\epsilon}{2}) \gamma \varphi(w)$ which is strictly smaller than $1-\frac{\epsilon}{4}$. As $\phi(u')=\frac{\epsilon}{8}$, the distance between $\psi(w)$ and $\psi(u')$ remains strictly smaller than $1-\frac{\epsilon}{8}$, thus the negative edge $wu'$ satisfies the condition.
\end{itemize}

We now consider the edges incident with $v'$. Observe that since $1 \leq \varphi(v)\leq 2-\frac{\epsilon}{2}$, we have $\gamma \leq \varphi'(v)\leq \gamma(2-\frac{\epsilon}{2})$. By the definition of $\psi$, and because $\varphi'(v)\geq 1-\frac{\epsilon}{8}$, we have:

 $$\gamma+\frac{\epsilon}{8} \leq \psi(v')= \varphi'(v) +\frac{\epsilon}{8}\leq \gamma(2-\frac{\epsilon}{2})+\frac{\epsilon}{8}=2-\frac{\epsilon}{8}.$$

As $v'$ is a copy of $v$, based on the sign of $wv$ we consider two cases.

\begin{itemize}
\item $wv'$ is a positive edge.

We need to show that the distance between $\psi(w)$ and $\psi(v')$ is at least $1$. Since $wv$ is a positive edge, it implies two possibilities: (1) $\varphi(w)\in [0, 1-\frac{\epsilon}{2}]$, (2) $\varphi(w)\in [2-\frac{\epsilon}{2},4-\epsilon]$. For case (1), $\varphi'(w)=\gamma \varphi(w)< 1-\frac{\epsilon}{8}$, then $\psi(w)=\varphi'(w)$ and thus the distance between $\psi(w)$ and $\psi(v')$ is larger than $\gamma+\frac{\epsilon}{8}$. For case (2), $\psi(w)=\varphi'(w)+\frac{\epsilon}{4}$. Thus the distance between $\psi(w)$ and $\psi(v')$ is at least $1+\frac{\epsilon}{8-2\epsilon}+\frac{\epsilon}{8}$, hence strictly larger than 1.

\item $wv'$ is a negative edge.

As $wv$ is a negative edge in $\hat{G}$, in any circular $(4-\epsilon)$-coloring $\varphi$, the distance of $\varphi(v)$ and $\varphi(w)$ is at most $1-\frac{\epsilon}{2}$ and then the distance of $\varphi'(v)$ and $\varphi'(w)$ is at most $\gamma(1-\frac{\epsilon}{2})<1-\frac{\epsilon}{4}$. Also, we have that $\frac{\epsilon}{2}\leq \varphi(w)\leq 3-\epsilon$. By the definition of $\psi$, if $\varphi'(w)\geq 1-\frac{\epsilon}{8}$, then $\psi(w)=\varphi'(w)+\frac{\epsilon}{4}$ and thus the distance between $\psi(w)$ and $\psi(v')$ is at most $\gamma(1-\frac{\epsilon}{2})+\frac{\epsilon}{8}<1-\frac{\epsilon}{8}$. It remains to show that if $\varphi'(w)< 1-\frac{\epsilon}{8}$, then the distance between $\psi(w)$ and $\psi(v')$ is smaller than $1-\frac{\epsilon}{8}$. In this case, $\psi(w)=\varphi'(w)$. Therefore, compared with the distance between $\varphi'(w)$ and $\varphi'(v)$, the distance between $\psi(w)$ and $\psi(v')$ is increased by $\frac{\epsilon}{8}$, therefore, it is at most $\gamma(1-\frac{\epsilon}{2})+\frac{\epsilon}{8}<1-\frac{\epsilon}{8}$.\qedhere
\end{itemize}

\end{proof}

\section{Signed bipartite planar simple graphs}\label{sec:BipartitePlanar}

In this section, we would like to prove Theorem~\ref{thm:SignedBipartitePlanar}. As in Section~\ref{sec:2degenerate}, we will first show that the circular chromatic number of any signed bipartite planar simple graph is strictly smaller than 4. Then we use the notion of tight cycle to get an improved upper bound. Finally, we show that this upper bound is tight.

To this end, we will work with a minimum counterexample. One of properties of a minimum counterexample follows from the following folding lemma of \cite{NRS13}. We recall that a plane graph or a signed plane graph is a (signed) planar graph together an embedding on the plane. For a plane graph, a \emph{separating} $l$-cycle is an $l$-cycle which is not a face.

\begin{lemma}\label{lem:Folding Lemma}{\em [Bipartite folding lemma]}
Let $\hat{G}$ be a signed bipartite plane graph and let $2k$ be the length of its shortest negative cycle. Let $F$ be a face whose boundary is not a negative cycle of length $2k$. Then there are vertices $v_{i-1}, v_i,v_{i+1}$, consecutive in the cyclic order of the boundary of $F$, such that identifying $v_{i-1}$ and $v_{i+1}$, after a possible switching at one of the two vertices, the result remains a signed bipartite plane graph whose shortest negative cycle is still of length $2k$.
\end{lemma} 

We observe that by applying this lemma repeatedly we get an image of $\hat{G}$ which is also a signed bipartite plane graph in which every facial cycle is a negative cycle of length exactly $2k$.

\begin{theorem}\label{thm:bipartite}
For any signed bipartite planar simple graph $(G, \sigma)$, we have $\chi_c(G, \sigma)< 4$.
\end{theorem}

\begin{proof}
Assume that $(G, \sigma)$ is a minimum counterexample, i.e.,  for no $\epsilon>0$, $(G, \sigma)$ admits a circular $(4-\epsilon)$-coloring, and $|V(G)|$ is minimized.

The minimality of $(G, \sigma)$, together with the bipartite folding lemma, implies that every facial cycle of $(G, \sigma)$, in any planar embedding of $G$, is a negative $4$-cycle. From here on, we will consider $(G, \sigma)$ together with a planar embedding. Moreover, since any subgraph of $(G, \sigma)$ is also a signed bipartite planar simple graph, it follows from Theorem~\ref{thm:2-vertex} that $\delta(G)\geq 3$.

We proceed by proving some structural properties of $(G, \sigma)$ in the form of claims.

\medskip
\noindent
{\bf Claim 1}
{\em Every vertex of even degree in $(G, \sigma)$ must be in a separating positive $4$-cycle.}

Assume to the contrary that a vertex $u$ is of even degree and it is in no separating positive $4$-cycle. Let $C$ be the boundary of the face in $(G-u, \sigma)$ which contains $u$. This cycle $C$ in the embedding of $(G, \sigma)$ bounds $d(u)$ faces, each of which is a negative 4-cycle. Thus $C$ is a positive cycle. Since switching does not affect the circular chromatic number, we may assume $\sigma$ is a signature in which all the edges of $C$ are positive. 

Let $(G', \sigma')$ be the signed graph obtained from $(G, \sigma)$ by contracting all edges incident with $u$ and by replacing each set of parallel edges of a same sign with a single edge of the same sign. Observe that, as $u$ is in no separating positive 4-cycle, $(G', \sigma')$ has no digon. Thus $(G',u')$ is a signed simple graph.  Furthermore, it is a signed bipartite planar simple graph which has less vertices than $(G, \sigma)$. Thus it admits a circular $(4-\epsilon)$-coloring for some positive $\epsilon$. But then Theorem~\ref{thm:F_u} implies that $\chi_c(G, \sigma)<4$.

\medskip
\noindent
{\bf Claim 2}
{\em For every pair of adjacent vertices each of an odd degree in $(G, \sigma)$, at least one is in a separating positive $4$-cycle.}

The proof of this claim is similar to the previous one. Towards the contradiction, let $x$ and $y$ be two adjacent vertices of odd degrees, neither of which is in a separating positive $4$-cycle. We consider the facial cycle $C$ which is obtained after deleting $x$ and $y$, and once again conclude that $C$ must be a positive cycle. Without loss of generality, we assume that $\sigma$ assigns positive signs to all edges of $C$. We may also assume that $xy$ is a negative edge.

We consider two signed graphs as follows. The first one is obtained from $(G-xy, \sigma)$ by contracting all the edges incident to $x$ where the new vertex is denoted $u$, and by contracting all the edges incident to $y$ where the new vertex is denoted $v$ and then adding a positive edge to connect $u$ and $v$. We denote the result by $\hat{G}'$ and note that it is a signed bipartite planar simple graph with no digon. By the minimality of $(G, \sigma)$, we conclude that $\chi_c(\hat{G}')<4$. The second signed graph we consider is obtained from $(G,\sigma)$ by identifying positive neighbors of $x$ into a new vertex $u$, the negative ones into a new vertex $u'$, and by identifying positive neighbors of $y$ into a new vertex $v$, the negative ones into a new vertex $v'$. Let $\hat{G}''$ be the resulting signed graph. We note that $\hat{G}''$ is not necessarily planar anymore. But regardless one can easily observe that $\hat{G}''$ is a subgraph of an $F_{uv} (\hat{G}')$. 

It follows from Theorem~\ref{thm:F_uv} that $\hat{G}''$ is in $\mathcal{C}_{<4}$, but $\hat{G}''$ is a homomorphic image of $(G, \sigma)$ and its circular chromatic number provides a bound on the circular chromatic number of $(G, \sigma)$.

\medskip
\noindent
{\bf Claim 3} 
{\em There is no positive $4$-cycle in $(G, \sigma)$.}

Since all faces of $(G, \sigma)$ are negative 4-cycles, any such a cycle must be a separating cycle. Towards a contradiction, among all separating positive 4-cycles, let $C$ be a separating positive 4-cycle with the minimum number of vertices inside. Let $v_1$, $v_2$, $v_3$, and $v_4$ be the four vertices of $C$ in this cyclic order. Let $u$ be a vertex inside $C$. As $(G,\sigma)$ is bipartite, $u$ can be adjacent to at most two vertices of $C$ and as $(G, \sigma)$ has minimum degree at least $3$, it must have a neighbor $v$ which is also inside $C$. By Claim~1 and Claim~2, at least one of $u$ or $v$, say $u$, is in a separating positive 4-cycle, denoted $C_u$. Since $C$ contains the minimum number of vertices inside, $C_u$ cannot be all inside of $C$, thus $u$ is adjacent to two vertices of $C$. Noting that $G$ is bipartite, and by symmetry, we may assume $v_1$ and $v_3$ are adjacent to $u$. 

We now claim that the vertex $v$ is in no separating positive 4-cycle. That is because, if so, then using the same argument, and noting the bipartiteness of $G$, it must be adjacent to $v_2$ and $v_4$. However, $v$ is separated from at least one of them by the path $v_1uv_3$ and $C$. Therefore, by Claim~1, $v$ is of odd degree and, by Claim~2, each neighbor of $v$ must be in a separating positive 4-cycle. As $v$ is of degree at least $3$ and can only be adjacent to at most one of $v_2$ and $v_4$, it has a neighbor $u_1$ which is in a separating 4-cycle. As $u_1$ is in the same part (of the bipartition of $G$) as $u$, and again by the minimality assumption on $C$, the vertex $u_1$ must be connected with both of $v_1$ and $v_3$. We note that this would separate $v$ from both $v_2$ and $v_4$. Thus there must be a third neighbor $u_2$ (of $v$), which is also in a separating 4-cycle and for the same reason then must be adjacent to both $v_1$ and $v_3$. The subgraph induced by $v$, $v_1$, $v_3$, $u$, $u_1$, and $u_2$ is then the complete bipartite graph $K_{3,3}$, contradicting that $G$ is a planar graph.
\medskip

To complete the proof of the theorem, we observe that, by Claims 1 and 3, all vertices must be of odd degree, and, by Claim 2, no two of them can be adjacent, but then any mapping to the points of any circle is circular coloring, a contradiction with our choice of $(G, \sigma)$.
\end{proof}

Next, using the notion of tight cycle, we improve the bound of Theorem~\ref{thm:bipartite}. We provide a concrete bound in terms of the number of vertices and then show that this improved bound is tight.

\begin{theorem}\label{thm:planarbipartite}
	For any signed bipartite planar simple graph $(G, \sigma)$ on $n$ vertices, we have:
	\begin{itemize}
		\item For each odd value of $n$, $\chi_c(G, \sigma)\leq  4-\dfrac{8}{n+1}$.
		\item For each even value of $n$, $\chi_c(G, \sigma)\leq  4-\dfrac{8}{n+2}$.
	\end{itemize}
Moreover, these bounds are tight for each value of $n\geq 2$.
\end{theorem}

\begin{proof}
As stated in Proposition~\ref{prop:st}, we know that $\chi_c(G, \sigma)=\frac{p}{q}$ where $p$ is twice the length of a cycle in $G$. As $G$ is a bipartite graph, the length of each cycle is even. Thus $p=4k$ for some positive integer $k$ such that $2k\leq n$.

 Since $\chi_c(G, \sigma)<4$ we have $\frac{p}{q}<4$, in other words, $4k<4q$. As $k$ and $q$ are integers we have $k+1\leq q$. Thus $\chi_c(G, \sigma)\leq \frac{4k}{k+1}=4-\frac{4}{k+1}$. The upper bounds claimed in the theorem then follows by noting that $n\geq 2k$ and that $n\geq 2k+1$ when $n$ is odd.

To prove that the bounds are tight we need to build an example $\Gamma^*_i$ for when $n=2i$ is even. Then by adding an isolated vertex to $\Gamma_{i}$ we get an example that works for $n=2i+1$.  

For $i\geq 2$, the signed graph $\Gamma^*_i$ is built from the signed graph  $\Omega_{i-1}$ by subdividing the edge $v_1v_2$ once and assigning a positive sign to one of the resulting edges and negative sign to the other. 

We should note that $\Gamma^*_i$ is a homomorphic image of the signed graph $\Gamma_i$ from Definition~43 of \cite{NWZ21}. To build $\Gamma^*_i$ from $\Gamma_i$ (of \cite{NWZ21}), one must identify the last two vertices (of $\Gamma_i$) after a suitable switching. In fact it follows that  $\Gamma^*_i$ is the core of $\Gamma_i$. 
	
Then for even values of $i$ the formula for the circular chromatic number of $\Gamma^*_i$ follows from Corollary~46 of \cite{NWZ21} by taking $G$ to be $K_2$. A similar computation can be done by taking $\Gamma_{2k+1}$ as an indicator $\mathcal{I}_{-}$ in Lemma~41 of \cite{NWZ21} with $I_{+}$ being free to choose. Applying this to $(K_2, -)$ then we get the formula for the circular chromatic number of $\Gamma_{2k+1}$.

An independent proof, quite similar to the proof of the circular chromatic number of $\Omega_i$, can be done by an added fact that in any circular $(4-\frac{8}{2i+2})$-coloring of $\Gamma^*_i$ the tight cycle is a Hamilton cycle.
\end{proof}

\section{Discussion and Questions}

In this work we have observed that bounding the circular chromatic number of a very restricted families of signed graphs can capture some of the most motivating problems in graph theory such as the 4-color theorem. 

Then by strengthening some results from \cite{NWZ21} we provided  improved bounds for two families of signed graphs: signed 2-degenerate simple graphs and signed bipartite planar simple graphs. 

We note that some of the well-known problems in circular coloring of graphs fit into this study by viewing a graph $G$ as a signed graph $(G,+)$ where all edges are positive. In particular, providing the best possible bound for the circular chromatic number of planar graphs of a given odd girth is one of main questions in graph theory which captures the 4CT, the Gr\"otzsch theorem, and the Jaeger-Zhang conjecture. 

Here we mention a few new questions that are based on the notion of the circular coloring of signed graphs.

\begin{Question}\label{que:PlanarSimple}
Given a signed planar simple graph $\hat{G}$, does there exist an $\epsilon=\epsilon(\hat{G})$ such that $\hat{G}$ admits a circular $(6-\epsilon)$-coloring?
\end{Question}

First example of a signed planar simple graph whose circular chromatic number is larger than 4 is given in \cite{KN21}. An example of signed planar simple graph whose circular chromatic number is $\frac{14}{3}$ is given in \cite{NWZ21}. The upper bound of $6$ follows from the fact that planar simple graphs are $5$-degenerate. The existence of any signed planar simple graph with circular chromatic number larger than $\frac{14}{3}$ is an open problem.

Restricted on the class of signed bipartite planar graphs and with an added negative girth condition (that is the length of a shortest negative cycle), we have the following question.

\begin{Question}\label{que:PlanrBipartiteGirht6}
Given a signed bipartite planar graph $\hat{G}$ of negative girth $6$, does there exist an $\epsilon=\epsilon(\hat{G})$ such that $\hat{G}$ admits a circular $(3-\epsilon)$-coloring?
\end{Question}

That every signed bipartite planar graph of negative girth at least $6$ admits a circular 3-coloring is recently proved in \cite{NW21}, noting that this proof uses the 4CT and some extensions of it. On the other hand, the best example of signed bipartite planar graph of negative girth $6$ we know has circular chromatic number $\frac{14}{5}$. It remains an open problem to build such signed graphs of circular chromatic number between $\frac{14}{5}$ and $3$.
 
We should mention that a negative answer to Question~\ref{que:PlanrBipartiteGirht6} would imply a negative answer to Question~\ref{que:PlanarSimple}. Let $T_2(G,\sigma)$ be a signed graph obtained from $(G, \sigma)$ by subdividing each edge $uv$ once and then assign a signature in such a way that the sign of the corresponding $uv$-path is $-\sigma(uv)$.
Viewing positive and negative paths of length 2 as $\mathcal{I}_{-}$ and $\mathcal{I}_{+}$ (respectively), and applying Lemma~41 of \cite{NWZ21} we have $$\chi_c(T_2(G,\sigma))=\frac{4\chi_c(G, \sigma)}{2+\chi_c(G, \sigma)}.$$

For signed bipartite planar graphs of negative girth at least $8$, the upper bound of $\frac{8}{3}$ for their circular chromatic numbers is proved in \cite{NPW21}. For signed bipartite planar graphs of negative girth $2k$, $k\geq 5$, the best current bound follows from recent results of \cite{LNWZ21}.

\medskip
{\bf Acknowledgement.} 
This work is supported by the French ANR project HOSIGRA (ANR-17-CE40-0022) and by the Slovak Science and Technology Assistance Agency under the contract No. APVV-19-0308. It has also received funding from the European Union's Horizon 2020 research and innovation program under the Marie Sklodowska-Curie grant agreement No 754362.

\bibliographystyle{plain}

\end{document}